\documentclass[11pt]{article}
\usepackage{amssymb,amsfonts,amsmath,amsthm}%
\usepackage{esint}%
\usepackage{graphicx}%
\newcommand{\figdraft}{false}%
\newcommand{\figfile}[1]{#1}%
\newcommand{\figwidth}{0.28\textwidth}%
\newcommand{\figstretch}{4.65}%

\theoremstyle{plain}%
\newtheorem{theorem}{Theorem}[]%
\newtheorem{lemma}[theorem]{Lemma}%
\newtheorem{assumption}[theorem]{Assumption}%
\newenvironment{Proof}[1][.]%
 {\begin{proof}[Proof #1.]}%
 {\end{proof}}
\newcommand{\ul}[1]{\underline{#1}}

\newcommand{\fspace}[1]{{\mathsf{#1}}}
\newcommand{\fspaceL}{\fspace{L}}

\newcommand{\ol}[1]{{\overline{#1}}}

\newcommand{\phase}{{\varphi}}

\newcommand{\Rset}{{\mathbb{R}}}
\newcommand{\Zset}{{\mathbb{Z}}}

\newcommand{\oointerval}[2]{(#1,\,#2)}%
\newcommand{\ccinterval}[2]{[#1,\,#2]}%

\newcommand{\DO}[1]{{O\at{#1}}}

\newcommand{\fin}{{\rm fin}}

\newcommand{\ini}{{\rm ini}}

\newcommand{\const}{{\rm const}}
\newcommand{\sym}{{\rm sym}}

\newcommand{\tdots}{{...}}%

\newcommand{\sympl}{{\rm sympl}}

\newcommand{\pair}[2]{{\left({#1},\,{#2}\right)}}
\newcommand{\skp}[2]{{\left\langle{#1},\,{#2}\right\rangle}}
\newcommand{\nskp}[2]{{\langle{#1},\,{#2}\rangle}}

\newcommand{\at}[1]{{\left({#1}\right)}}
\newcommand{\nat}[1]{(#1)}
\newcommand{\bat}[1]{{\big(#1\big)}}

\newcommand{\triple}[3]{{\left({#1},\,{#2},\,{#3}\right)}}

\newcommand{\bigpar}{\par\quad\newline\noindent}

\newcommand{\wt}[1]{{\widetilde{#1}}}
\newcommand{\wh}[1]{{\widehat{#1}}}

\newcommand{\jump}[1]{{|\![#1]\!|}}

\newcommand{\norm}[1]{\|{#1}\|}

\newcommand{\dint}[1]{\,\mathrm{d}#1}

\newcommand{\La}{{\Lambda}}

\newcommand{\ga}{{\gamma}}

\newcommand{\eps}{{\varepsilon}}

\newcommand{\la}{{\lambda}}
\newcommand{\si}{{\sigma}}

\newcommand{\om}{{\omega}}

\newcommand{\MiLagr}{j}

\newcommand{\calA}{\mathcal{A}}
\newcommand{\calB}{\mathcal{B}}
\newcommand{\calC}{\mathcal{C}}

\newcommand{\calF}{\mathcal{F}}

\newcommand{\calH}{\mathcal{H}}
\newcommand{\calI}{\mathcal{I}}

\newcommand{\calN}{\mathcal{N}}

\newcommand{\calP}{\mathcal{P}}

\newcommand{\calT}{\mathcal{T}}

\newcommand{\calW}{\mathcal{W}}
\newcommand{\calX}{\mathcal{X}}
\newcommand{\calY}{\mathcal{Y}}
\newcommand{\refcite}[1]{\cite{#1}}
\newcommand{\fncite}[1]{$^\text{\cite{#1}}$}
\usepackage[%
a4paper,%
nohead,%
twoside,
ignorefoot,
scale=0.75,
bindingoffset=.35cm
]{geometry}%
\usepackage{float}%
\begin{document}%
%
%
\title{Oscillatory Waves in Discrete Scalar Conservation Laws}%
\date{\today}%
\author{%
Michael Herrmann%
\thanks{
    Saarland University, Department of Mathematics,
    {\tt{michael.herrmann@math.uni-sb.de}}
}}%
\maketitle
%
%
%
\begin{abstract}%
We study Hamiltonian difference schemes for scalar conservation laws with monotone flux
function and establish the existence of a three-parameter family of periodic travelling waves
(wavetrains). The proof is based on an integral equation for the dual wave profile and
employs constrained maximization as well as the invariance properties of a gradient flow. We
also discuss the approximation of wavetrains and present some numerical results.
\end{abstract}%
%
%
\quad\newline\noindent%
\begin{minipage}[t]{0.15\textwidth}%
Keywords: %
\end{minipage}%
\begin{minipage}[t]{0.8\textwidth}%
 \emph{conservation laws}, %
 \emph{difference schemes}, %
 \emph{dispersive shocks},  \\%
 \emph{travelling waves}, %
 \emph{Hamiltonian lattices}, %
 \emph{variational integrators} %
\end{minipage}%
\medskip
\newline\noindent
\begin{minipage}[t]{0.15\textwidth}%
MSC (2000): %
\end{minipage}%
\begin{minipage}[t]{0.8\textwidth}%
35L60, 
37K60, 
47J30  
\end{minipage}%
%
%
%
%
%
\setcounter{tocdepth}{5} %
\setcounter{secnumdepth}{4}
{\scriptsize{\tableofcontents}}%
%
%

\section{Introduction}\label{sec:intro}
%
This paper is concerned with oscillatory patterns in the nonlinear lattice equation
\begin{align}
\label{Intro.Lattice}
2\,\dot{u}_\MiLagr+\Phi^\prime\at{u_{j+1}}-\Phi^\prime\at{u_{j-1}}=0,\qquad\qquad{u}\in\Rset,
\quad{ t\geq0 } , \quad { j } \in\Zset ,
\end{align}
which is, up to an appropriate scaling, a \emph{centred} difference scheme for the scalar
conservation law
\begin{align}
\label{Intro.PDE1}
\partial_\tau\bar{u}+\partial_\xi\Phi^\prime\at{\bar{u}}=0,\qquad\qquad\bar{u}\in\Rset,
\quad\tau\geq0, \quad\xi\in\Rset.
\end{align}
Although the lattice cannot be used for the approximate computation of (non-smooth) solutions
of \eqref{Intro.PDE1}, there exist several reasons why it seems worth investigating the
dynamical properties of \eqref{Intro.Lattice} in greater detail. First, both the lattice and
the PDE have the \emph{same} Hamiltonian structure, and \eqref{Intro.Lattice} can be regarded
as a variational integrator for \eqref{Intro.PDE1} with respect to the space discretization,
see appendix \ref{sec:appendix}. Studying \eqref{Intro.Lattice} therefore allows to
understand in which aspects quasilinear Hamiltonian PDEs differ from their spatially discrete
counterparts. Second, the lattice is, similar to Korteweg-de Vries-type (KdV) equations, a
\emph{dispersive regularization} of \eqref{Intro.PDE1}, and hence it generates
\emph{dispersive shocks} instead of Lax shocks. These dispersive shocks describe the
fundamental mode of self-thermalization in dispersive Hamiltonian systems but are well
understood only for systems that are completely integrable. The ODE system
\eqref{Intro.Lattice} provides a class of \emph{non-integrable} examples which can be
simulated effectively, for instance using variational integrators for the time
discretization. Finally, discrete conservation laws provide toy models for more complicate
Hamiltonian lattices. Fermi-Pasta-Ulam (FPU) chains, for instance, are equivalent to
difference schemes for the so called $p$-system, which is a nonlinear hyperbolic system of
two conservation laws.
\par
In what follows we focus on a particular aspect of the lattice dynamics and investigate
special coherent structures, namely \emph{wavetrains}. These are periodic travelling wave
solutions to \eqref{Intro.Lattice} and therefore linked to nonlinear
advance-delay-differential equations.
\par
We emphasize that all subsequent considerations require a \emph{centred} difference operator
in \eqref{Intro.Lattice} as only this one gives rise to a Hamiltonian lattice. Other discrete
scalar conservation laws, as for instance upwind schemes, have different dynamical properties
and are excluded.
%
%
\subsection{Dispersive shocks and wavetrains}
%

The relation between \eqref{Intro.Lattice} and \eqref{Intro.PDE1} manifests under the
\emph{hyperbolic scaling}
\begin{align}%
\label{Intro.Scaling}
\tau=\eps{t},\qquad\xi=\eps{j},\qquad
u_j\at{t}=\bar{u}\at{\eps{t},\eps{j}},
\end{align}
where $0<\eps\ll1$ is the scaling parameter, and $\tau$ and $\xi$ denote the
\emph{macroscopic} time and particle index, respectively. On a formal level we can expand
differences in powers of differential operators, i.e.,
\begin{align*}%
\bar{u}\at{\xi+\eps}-\bar{u}\at{\xi-\eps}=
2\partial_\xi{\bar{u}}\at{\xi}+
\tfrac{1}{3}\eps^2\partial_\xi^3{\bar{u}}\at{\xi}+\DO{\eps^5},
\end{align*}%
so that, to leading order in $\eps$, the lattice dynamics is governed by the KdV-type PDE
\begin{align}
\notag
\partial_\tau{\bar{u}}+\at{\partial_\xi+
\tfrac{1}{6}\eps^2\partial_\xi^3}\Phi^\prime\at{\bar{u}}=0,
\end{align}
which is in fact a dispersive regularization of \eqref{Intro.PDE1}.
\par
\begin{figure}[t!]%
\centering{%
\renewcommand{\arraystretch}{\figstretch}%
\begin{tabular}{lcr}%
\begin{minipage}{\figwidth}%
\includegraphics[width=\textwidth, draft=\figdraft]%
{\figfile{ds_1_11}}%
\end{minipage}%
&%
\begin{minipage}{\figwidth}%
\includegraphics[width=\textwidth, draft=\figdraft]%
{\figfile{ds_1_12}}%
\end{minipage}%
&
\begin{minipage}{\figwidth}%
\includegraphics[width=\textwidth, draft=\figdraft]%
{\figfile{ds_1_13}}%
\end{minipage}%
\\%
\begin{minipage}{\figwidth}%
\includegraphics[width=\textwidth, draft=\figdraft]%
{\figfile{ds_1_14}}%
\end{minipage}%
&%
\begin{minipage}{\figwidth}%
\includegraphics[width=\textwidth, draft=\figdraft]%
{\figfile{ds_1_15}}%
\end{minipage}%
&
\begin{minipage}{\figwidth}%
\includegraphics[width=\textwidth, draft=\figdraft]%
{\figfile{ds_1_16}}%
\end{minipage}%
\end{tabular}%
}
\caption{%
Numerical lattice solution with periodic boundary conditions for the data from
\eqref{Intro.Example1}: Snapshots of $u_j$ against the macroscopic particle index $j/N$ at
several macroscopic times $\tau$; the vertical lines in the lower right picture mark the
positions for the magnifications in Figure~\ref{Fig:Intro.2}. \emph{Interpretation:} Instead
of Lax shocks the lattice generates dispersive shocks with strong microscopic oscillations.
}%
\label{Fig:Intro.1}
\vspace{0.02\textwidth}%
\centering{%
\renewcommand{\arraystretch}{\figstretch}%
\begin{tabular}{lcr}%
\begin{minipage}{\figwidth}%
\includegraphics[width=\textwidth, draft=\figdraft]%
{\figfile{ds_1_21}}%
\end{minipage}%
&%
\begin{minipage}{\figwidth}%
\includegraphics[width=\textwidth, draft=\figdraft]%
{\figfile{ds_1_22}}%
\end{minipage}%
&%
\begin{minipage}{\figwidth}%
\includegraphics[width=\textwidth, draft=\figdraft]%
{\figfile{ds_1_23}}%
\end{minipage}%
\end{tabular}%
}
\caption{%
Magnifications of the oscillations from Figure~\ref{Fig:Intro.1} in three selected points:
Black points are numerical data, Grey curves represent interpolating splines and are drawn
for better illustration. \emph{Interpretation:} The microscopic oscillations can be described
by a modulated travelling wave. In other words, the local oscillations are generated by a single
wavetrain whose parameters change on the macroscopic scale.
}%
\label{Fig:Intro.2}
\end{figure}%
As already mentioned, a key feature of all dispersive regularizations of \eqref{Intro.PDE1}
are dispersive shocks.\fncite{HL91,El05,AMSMSP:DHR,MX10} For illustration, and to motivate
our analytical investigations, we now describe the formation of such dispersive shocks in
numerical simulations of \eqref{Intro.Lattice}. For convenience we shall consider spatially
periodic solutions with $u_j=u_{j+N}$ for some $N\gg1$, which then defines the natural
scaling parameter $\eps=1/N$. Starting with long-wave-length initial data
\begin{align*}
u_j\at{0}=\bar{u}_\ini\at{j/N},
\end{align*}
where $\bar{u}_\ini$ is a smooth and $1$-periodic function in $\xi$, we can solve
\eqref{Intro.Lattice} numerically by using the variational integrator
\eqref{App:Symplectic.Integrator}. A typical example is depicted in Figure~\ref{Fig:Intro.1}
for the data
\begin{align}
\label{Intro.Example1}
\Phi\at{u}=\tfrac{1}{2}u^2+\tfrac{1}{4}u^4
,\quad
N=4000%
,\quad\bar{u}_\ini\at{\xi}=1-\tfrac{13}{20}\sin\at{2\pi\xi}.
\end{align}
For small macroscopic times $\tau$, we observe that the data remain in the long-wave-length
regime, so we can expect that the lattice solutions converge as $\eps\to0$ in some strong
sense to a smooth solution of \eqref{Intro.PDE1}. This convergence is not surprising in view
of Strang's theorem\fncite{Str64}, see also the discussion in Ref.~\refcite{GL88,HL91}. At
time $\tau\approx0.04$, however, the solution to the PDE \eqref{Intro.PDE1} forms a Lax
shock, which then propagates with speed $s$ according to the Rankine-Hugoniot jump condition
$s\jump{\bar{u}}=\jump{\Phi^\prime\at{\bar{u}}}$. After the onset of the shock singularity,
the lattice solutions do not converge anymore to a solution of \eqref{Intro.PDE1}, not even
in a weak sense, but exhibit strong microscopic oscillations. These oscillations spread out
in space and time, and constitute the dispersive shock. This phenomenon is fascinating from
both the mathematical and the physical point of view because the formation of dispersive
shocks can be regarded as a self-thermalization of the nonlinear Hamiltonian system. We refer
to Ref.~\refcite{DH08,HR10a} for more details including a thermodynamic discussion of
dispersive FPU shocks.
\par%
The observation that certain difference schemes for hyperbolic conservation laws produce
dispersive shocks is not new, see for instance Ref.~\refcite{Lax86,GL88}, which investigate
dispersive shocks in
\begin{align}
\notag
2\,\dot{v}_\MiLagr+v_j\at{v_{j+1}-v_{j-1}}=0.
\end{align}
This lattice also belongs to the class of equations considered here as it transforms via
$u_j=\ln{v_j}$ into \eqref{Intro.Lattice} with $\Phi^\prime=\exp$, which is the completely
integrable Kac-von Moerbeke lattice\fncite{KvM75}.
\bigpar
The key observation in our context is that the oscillations within a dispersive shock exhibit
the typical behaviour of a modulated oscillation. As illustrated in Figure~\ref{Fig:Intro.2},
the local oscillations in the vicinity of a given macroscopic point $\pair{\tau}{\xi}$
resemble a periodic profile function, and we can expect this profile to be generated by a
single wavetrain. The parameters of this wavetrain, however, depend on $\tau$ and $\xi$,
which can be seen from the fact that each magnification in Figure~\ref{Fig:Intro.2} displays
a another amplitude, wave number, and average.
\par
The goal of this paper is to prove the existence of a three-parameter family of wavetrains
for a huge class of nonlinear potentials $\Phi$. Since the only crucial assumption we have to
make is strict convexity of $\Phi$, our results cover a large number of non-integrable
variants of \eqref{Intro.Lattice}. Due to rigorous results for integrable systems, such as
the KdV equation\fncite{LL83,LLV93} and the Toda chain\fncite{Kam91,DM98}, we expect the
parameter modulation within a dispersive shock to be governed by a variant of Whitham's
modulation equations.\fncite{Whi65,Whi74} The formal derivation and investigation of this
Whitham system, however, is left for future research. We also do not justify that the
oscillations within dispersive shocks take in fact the form of modulated wavetrains. For a
rigorous proof in non-integrable systems we still lack the analytical tools; a reliable
numerical justification would be possible, and was carried out for FPU in
Ref.~\refcite{DH08}, but is beyond the scope of this paper.
\bigpar
A travelling wave is a special solution to \eqref{Intro.Lattice} with
\begin{align}
\label{TWAnsatz.U}
u_j\at{t}=U\at{kj-\om{t}},
\end{align}
where $k$ is the \emph{wave number}, $\omega$ is the \emph{frequency}, and the \emph{profile}
$U$ depends on the \emph{phase variable} $\phase=kj-\omega{t}$. In this paper we are solely
interested in wavetrains, which have periodic $U$, but mention that one might also study
\emph{solitons} and \emph{fronts} having homoclinic and heteroclinic profiles, respectively.
\par
Splitting $U$ into its constant and zero average part via $U\at\phase=v+V\at\phase$, we infer
from \eqref{Intro.Lattice} that each wavetrain satisfies
\begin{align}
\label{TWEquation.Version1}
\om\tfrac{\dint}{\dint{\phase}}V=\nabla_k\Phi^\prime\at{v+V},
\end{align}
where $\nabla_k$ is defined by
\begin{align}
\notag
\at{\nabla_k{P}}\at\phase=\tfrac{1}{2}\bat{P\at{\phase+k}-P\at{\phase-k}}.
\end{align}
%
%
\subsection{Main result and organisation of the paper}
%
Since the wavetrain equation \eqref{TWEquation.Version1} involves both advance and delay
terms there is no notion of an initial value problem, and one has to use rather sophisticated
methods to establish the existence of solutions. Possible candidates, which have proven to be
powerful for other Hamiltonian lattices, are rigorous perturbation arguments\fncite{FP99},
spatial dynamics with centre manifold reduction\fncite{Ioo00,IJ05}, critical point
techniques\fncite{SW97,PP00,SZ07}, and constrained optimization\fncite{FW94,FV99,Her10a}.
\par
Our approach also exploits the underlying variational structure and restates
\eqref{TWEquation.Version1} as
\begin{align}
\label{TWEquation.Abstract}
\si\Psi^\prime\at{Q}-\si\eta=\calB{Q},
\end{align}
where $Q$ is the \emph{dual profile} to be introduced in \S\ref{sec:prelim.Dual}. Moreover,
$\calB$ is a compact and symmetric integral operator, and $\Psi$ is the \emph{dual
potential}, i.e., the Legendre transform of $\Phi$. This formulation allows to construct
wavetrains as solutions to a constrained optimization problem, where $\si$ and $-\si\eta$
play the role of Lagrangian multipliers.
\par
The existence proof for wavetrains given below is based on a combination of variational
arguments (direct method) and dynamical concepts (invariant sets of flows). These ideas can
also be applied to other Hamiltonian lattices and different types of coherent
structures\fncite{Her09,HR10b,Her10c}, but discrete scalar conservation laws are special
since they are first order in time and real-valued. All other Hamiltonian lattices we are
aware of are either second order in time or vector-valued, and allow for a simpler
variational setting for travelling waves.
\par
Our main result guarantees the existence of a three-parameter family of wavetrains and can be
summarized as follows.
\par\noindent%
\begin{minipage}{\textwidth}
\bigskip {\bf Main result. }
\emph{Suppose that $\Phi$ is strictly convex and satisfies some regularity
\mbox{assumptions}. Then, for each $k\in\oointerval{0}{\pi}$ there exists a two-parameter
family of $2\pi$-periodic wavetrains with $v\in\Rset$ and frequency $\om>0$ such that the
profile $V$ is even and unimodal on $\ccinterval{-\pi}{\pi}$.}
\medskip%
\end{minipage}
\begin{figure}[t!]
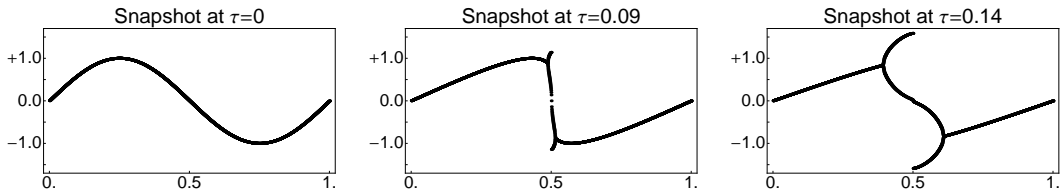
%
\centering{%
\renewcommand{\arraystretch}{\figstretch}%
\begin{tabular}{lcr}%
\begin{minipage}{\figwidth}%
\includegraphics[width=\textwidth, draft=\figdraft]%
{\figfile{ds_2_11}}%
\end{minipage}%
&%
\begin{minipage}{\figwidth}%
\includegraphics[width=\textwidth, draft=\figdraft]%
{\figfile{ds_2_12}}%
\end{minipage}%
&%
\begin{minipage}{\figwidth}%
\includegraphics[width=\textwidth, draft=\figdraft]%
{\figfile{ds_2_13}}%
\end{minipage}%
\end{tabular}%
}
\caption{%
Numerical lattice solution for $N=4000$, $\Phi^\prime\at{u}=u^2$ and
$\bar{u}_\ini\at\xi=\sin\at{2\pi\xi}$. \emph{Interpretation:} If $\Phi^{\prime\prime}$
changes sign, the lattices can generate modulated binary oscillation instead of dispersive
shocks.
}%
\label{Fig:Intro.3}
\end{figure}%
\par
The assumptions on $\Phi$ will be specified in Assumption \ref{Main.Ass} and the precise
existence result is formulated in Theorem \ref{Theo:Main.Result}. Here we proceed with some
comments concerning the choice of $k$ and the convexity of $\Phi$.
\begin{enumerate}
\item
Since \eqref{TWAnsatz.U} is invariant under
$\pair{k}{\om}\rightsquigarrow\pair{-k}{-\om}$, there is a similar result for
$k\in\oointerval{-\pi}{0}$ with $\om<0$. The cases $k=0$ and $k=\pm\pi$, however, are
degenerate, see \S\ref{sec:prelim.Elem}.
\item
As time reversal in \eqref{Intro.Lattice} corresponds to
$\pair{k}{\om}\rightsquigarrow\pair{k}{-\om}$, the existence result can easily be adapted
to the case of strictly concave $\Phi$.
\item
Our proof does not cover potentials $\Phi$ that change from convex to concave or vice
versa, and it is not obvious what happens near zeros of $\Phi^{\prime\prime}$. Numerical
simulations as presented in Figure~\ref{Fig:Intro.3}, however, indicate that then a
three-parameter families of wavetrains might not exist anymore.
\end{enumerate}
The paper is organized as follows. In \S\ref{sec:prelim} we summarize some elementary
properties of wavetrains. We also derive the integral equation for the dual profile and
discuss some normalization which allows to simplify the presentation. \S\ref{sec:proof} is
devoted to the proof of the existence result. To point out the key idea we first state an
abstract result in Theorem \ref{Theo.Existence.Abstract}, and show afterwards that the
assertions we made are satisfied for wavetrains, see Theorem \ref{Theo:Main.Result}. Finally,
in \S\ref{sec:num} we compute wavetrains by means of a discrete gradient flow on a constraint
manifold.
%
\section{Preliminaries about discrete conservation laws}\label{sec:prelim}
%
In this section we summarize some elementary properties of wavetrains, derive the integral
equation \eqref{TWEquation.Abstract} for the dual profile, and introduce some normalization.
%
%
%
\subsection{Elementary properties and special solutions}\label{sec:prelim.Elem}
%
%
First we observe that the wavetrain equation \eqref{TWEquation.Version1} is invariant under
shifts $\phase\rightsquigarrow\phase_0$, reflections $\phase\rightsquigarrow-\phase$, and
scalings
\begin{align*}
k\rightsquigarrow{k/s}
,\qquad
\om\rightsquigarrow\om/s
,\qquad
V\at\phase\rightsquigarrow{V\at{s\phase}}
,\qquad
v\rightsquigarrow{v}
\end{align*}
with $s\neq{0}$. It is therefore sufficient to consider a fixed periodicity cell $\La$ and
wave numbers $k\in\La$, so from now on we assume
\begin{align*}
\La=\ccinterval{-\pi}{\pi}.
\end{align*}
Recall that a periodic function $V:\La\to\Rset$ is \emph{even} if $V\at\phase=V\at{-\phase}$
holds for all $\phase\in\La$, and that an even function on $\La$ is \emph{unimodal} if it is
monotone on $\ccinterval{-\pi}{0}$.
\par
Second, we notice that \eqref{TWEquation.Version1} degenerates for both $k=0$ and $k=\pm\pi$.
In fact, wavetrains for $k=0$ are globally constant with $u_j\at{t}\equiv{u_1}\at{0}$, while
wavetrains for $k=\pi$ are stationary binary oscillations with $u_{2j}\at{t}={u_0}\at{0}$ and
$u_{2j+1}\at{t}={u_1}\at{0}$. Therefore, and since \eqref{TWEquation.Version1} is invariant
under $\pair{k}{\om}\rightsquigarrow\pair{-k}{-\om}$, we restrict all subsequent
considerations to $k\in\oointerval{0}{\pi}$.
\bigpar
\begin{figure}[t!]
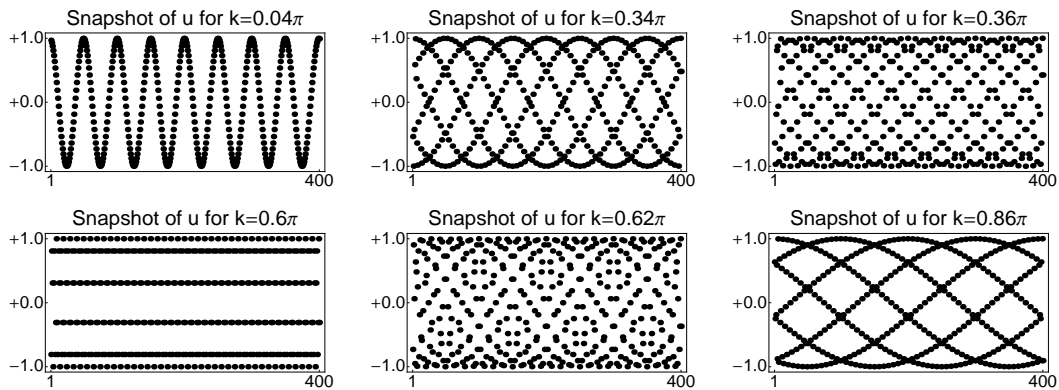
%
\centering{%
\renewcommand{\arraystretch}{\figstretch}%
\begin{tabular}{lcr}%
\begin{minipage}{\figwidth}%
\includegraphics[width=\textwidth, draft=\figdraft]%
{\figfile{ex_0_force_1}}%
\end{minipage}%
&%
\begin{minipage}{\figwidth}%
\includegraphics[width=\textwidth, draft=\figdraft]%
{\figfile{ex_0_force_2}}%
\end{minipage}%
&%
\begin{minipage}{\figwidth}%
\includegraphics[width=\textwidth, draft=\figdraft]%
{\figfile{ex_0_force_3}}%
\end{minipage}%
\\%
\begin{minipage}{\figwidth}%
\includegraphics[width=\textwidth, draft=\figdraft]%
{\figfile{ex_0_force_4}}%
\end{minipage}%
&%
\begin{minipage}{\figwidth}%
\includegraphics[width=\textwidth, draft=\figdraft]%
{\figfile{ex_0_force_5}}%
\end{minipage}%
&%
\begin{minipage}{\figwidth}%
\includegraphics[width=\textwidth, draft=\figdraft]%
{\figfile{ex_0_force_6}}%
\end{minipage}%
\end{tabular}%
}%
\caption{%
Snapshots of harmonic wavetrains for several values of $k$. Although the carrier profile is
always $U\at\phase=\cos\phase$, the plots $u_j=U\at{kj}$ against $j\in\Zset$ exhibit rather
different spatial patterns.
}%
\label{Fig:Harm}
\end{figure}%
For linear flux $\Phi^\prime$ we can solve \eqref{TWEquation.Version1} by Fourier transform.
Specifically, linearizing \eqref{TWEquation.Version1} around $v$ and restricting to unimodal
and even profiles we find that $\om$ depends on $k$ and $v$ via the dispersion relation
\begin{align*}
\om=\Omega\pair{k}{v}=\Phi^{\prime\prime}\at{v}\sin{k},
\end{align*}
and that the unique profile is $V\at\phase=\alpha\cos{\phase}$, where the amplitude
$\alpha\in\Rset$ is the third independent parameter. Harmonic wavetrains furthermore
exemplify that simple profiles can generate rather complex patterns for the spatial
oscillations in the lattice. This is illustrated in Figure~\ref{Fig:Harm}, and in
Figure~\ref{Fig:Snapshots} for the nonlinear case.
\par
The second case in which we can solve \eqref{TWEquation.Version1} explicitly are wavetrains
with wave number $k=\pi/2$.
\begin{lemma}
\label{Lem:ExplicitODEsolution}%
Each wavetrain for $k=\pi/2$ satisfies the Hamiltonian ODE
\begin{align}
\label{Lem:ExplicitODEsolution.Eqn1}%
\om\tfrac{\dint}{\dint\phase}V\at\phase=
+{\Phi}_{v,\,\sym}^\prime\bat{\tilde{V}\at{\phase}}
,\quad
\om\tfrac{\dint}{\dint\phase}\tilde{V}\at\phase=-
{\Phi}_{v,\,\sym}^\prime\bat{V\at{\phase}},
\end{align}
where
\begin{align*}
\tilde{V}\at\phase=V\at{\phase+\pi/2},
\quad%
{\Phi}_{v,\,\sym}\at{x}=\tfrac{1}{2}\bat{\Phi\at{v+x}+\Phi\at{v-x}}.
\end{align*}
In particular, if $\Phi$ is smooth and convex, then for each $v$ there exists a unique
one-parameter family of wavetrains with unimodal and even profiles that is parametrized by
$\varrho=\Phi_{v,\,\sym}\at{V\at{0}}$.
\end{lemma}
\begin{proof}
By a direct computation we infer from \eqref{TWEquation.Version1} that
\begin{align*}
\om\tfrac{\dint}{\dint{\phase}}V\at{\phase}
=%
\tfrac{1}{2}\Phi^\prime\bat{v+\tilde{V}\at{\phase}}-
\tfrac{1}{2}\Phi^\prime\bat{v+\tilde{V}\at{\phase-\pi}}=
-\om\tfrac{\dint}{\dint{\phase}}V\at{\phase+\pi}.
\end{align*}
Combining this with $\int_\La{V}\dint\phase=0$ we find
\begin{align*}
V\at{\phase+\pi}=-V\at{\phase}\qquad\forall\;\phase\in\Rset,
\end{align*}
and hence \eqref{Lem:ExplicitODEsolution.Eqn1}. Finally, since
\eqref{Lem:ExplicitODEsolution.Eqn1} is a planar Hamiltonian ODE the remaining assertions
follow from standard arguments.
\end{proof}
%
\subsection{Integral equation for the dual profile and normalization}\label{sec:prelim.Dual}
%
%
We define the \emph{dual profile} $Q$ of a wavetrain by
\begin{align}
\label{DualProblem.Trafo}
\Phi^\prime\at{v+V}=q+Q,
\qquad%
q=\fint_\La{\Phi^\prime\at{v+V}}\,\dint\phase,
\end{align}
where $\fint$ abbreviates the integral mean value, i.e.,
\begin{align*}
\fint_\La{V}\,\dint{\phase}=\tfrac{1}{2\pi}\int_{-\pi}^\pi{V}\at\phase\,\dint\phase.
\end{align*}
The key ingredient to our variational existence proof is to restate
\eqref{TWEquation.Version1} as an equation for $Q$. To this end we consider the Legendre
transform $\Psi$ of $\Phi$, which is well-defined and strictly convex provided that $\Phi$ is
strictly convex\fncite{DiB02}, and satisfies
\begin{align*}
\Psi\at{\zeta}=\sup_\nu
\bat{\zeta\nu-\Phi\at{\nu}},
\qquad\Phi^\prime\circ\Psi^\prime=\Psi^\prime\circ\Phi^\prime=\mathrm { id } .
\end{align*}
Using $\Psi$ and \eqref{DualProblem.Trafo} we now find that \eqref{TWEquation.Version1} is
equivalent to
\begin{align*}%
\om\tfrac{\dint}{\dint\phase}\Psi^\prime\at{q+Q}=\nabla_k{Q},
\end{align*}%
and integration with respect to $\phase$ yields
\begin{align}%
\label{DualProblem.Formulation}
\om\at{\Psi^\prime\at{q+Q}-v}=\calA_k{Q},\qquad%
v=
\fint_\La{\Psi^\prime\at{q+Q}}\,\dint\phase.
\end{align}%
Here, $v$ appears as a constant of integration and the integral operator $\calA_k$ is defined
by
\begin{align}
\label{Def.Operator.A}
\at{\calA_k{Q}}\at\phase=\frac{1}{2}\int_{\phase-k}^{\phase+k}
Q\at{\tilde\phase}\,\dint\tilde\phase.
\end{align}
With \eqref{DualProblem.Formulation} we have derived the dual formulation of
\eqref{TWEquation.Version1}; the main mathematical difference between both formulations will
be discussed at the end of \S\ref{sec:proof.I}.
\bigpar
For the existence proof in \S\ref{sec:proof} it is convenient to normalize
\eqref{DualProblem.Formulation} in two steps. At first we incorporate the parameter $q$ into
the nonlinearity by considering the normalized dual potential
\begin{align}
\label{Dual.Pot}
\Psi_q\at\zeta=\Psi\at{q+\zeta}-\Psi^\prime\at{q}\zeta-\Psi\at{q},
\end{align}
which satisfies
\begin{align*}
\Psi^{\prime\prime}_q\at\zeta=\Psi^{\prime\prime}\at{q+\zeta},\qquad
\Psi_q^\prime\at{0}=\Psi_q\at{0}=0.
\end{align*}
This transforms \eqref{DualProblem.Formulation} into
\begin{align*}%
\om\at{\Psi_q^\prime\at{Q}-\eta}=\calA_k{Q},\qquad%
\eta=v-\Psi^\prime\at{q}=
\fint_\La{\Psi_q^\prime\at{Q}}\,\dint\phase.
\end{align*}
The second normalization step is motivated by the harmonic case and the observation that the
formula for the \emph{phase speed} $\si$ depends on the value of $k$. In fact, the
linearization of \eqref{DualProblem.Formulation} around $q$ gives the dispersion relation
\begin{align*}
Q\at\phase{}=
\alpha\cos\phase,\quad\om=
\Omega\pair{q}{k}=\frac{\sin\at{k}}{\Psi^{\prime\prime}\at{q}}
,
\qquad{v}=\Psi^\prime\at{q},
\end{align*}
and we conclude that $\si=\om/k$ for $0<k\leq\pi/2$ but $\si=\om/\at{\pi-k}$ for
$\pi/2\leq{k}<\pi$, see Figure \ref{Fig_1}. We also notice that $\calA_k$ has slightly
different properties for $0<k\leq\pi/2$ and $\pi/2\leq{k}<\pi$. In particular, by
\eqref{Def.Operator.A} we have
\begin{align}
\label{Operator.A.Symmetry}
\calA_{\pi-k}Q=-\calA_k\calT{Q}
\end{align}
for all $Q$ with $\fint_\La{Q}\dint\phase=0$, where $\calT$ denotes the shift operator
\begin{align}
\notag
\calT{Q}=Q\at{\cdot+\pi}.
\end{align}
\begin{figure}[t!]%
\centering{
\includegraphics[width=.5\textwidth]{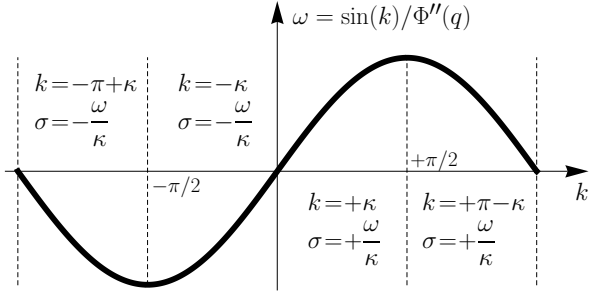}%
}%
\caption{%
Dispersion relation for the harmonic case and the different relations between wave number
$k$, frequency $\omega$, and phase speed $\sigma$. Thanks to the symmetry
$\triple{k}{\om}{\si}\leftrightsquigarrow\triple{-k}{-\om}{\si}$ it is sufficient to consider
$k\in\pair{0}{\pi}$.
}%
\label{Fig_1}%
\end{figure}%
To complete the normalization we now replace $\om$ by $\si$, and $\calA_k$ by either
$\wh{\calB}_k$ or $\wt{\calB}_{\pi-k}$, which are defined for $0<\kappa\leq\pi/2$ by
\begin{align}
\notag
\wh{\calB}_\kappa=\kappa^{-1}\calA_\kappa,\qquad
\wt{\calB}_\kappa=\kappa^{-1}\calA_{\pi-\kappa}=-\kappa^{-1}\calA_\kappa\calT.
\end{align}
Specifically, for $0<k\leq\pi/2$ we restate the dual wavetrain equation
\eqref{DualProblem.Formulation} as
\begin{align}
\label{Dual.Problem.I}
\si\at{\Psi_q^\prime\at{Q}-\eta}=\wh{\calB}_\kappa{Q}
,\qquad
\kappa=k,\qquad\sigma=\om/k,
\end{align}
whereas for $\pi/2\leq{k}<\pi$ we write
\begin{align}
\label{Dual.Problem.II}
\si\at{\Psi_q^\prime\at{Q}-\eta}=\wt{\calB}_\kappa{Q},\qquad
\kappa=\pi-{k},\qquad\sigma=\om/\at{\pi-k}.
\end{align}
Notice that for $k=\pi/2$ both formulations agree and encode the ODE solution from
Lemma~\ref{Lem:ExplicitODEsolution}. In particular, in this case we have the further symmetry
$Q=-\calT{Q}$.
%
\section{Variational existence proof for wavetrains}\label{sec:proof}
%
%
In this section we prove the existence of a three-parameter family of wavetrains by showing
that the dual integral equations \eqref{Dual.Problem.I} and \eqref{Dual.Problem.II} possess
corresponding three-parameter families of solutions. To elucidate the key ideas we start with
an abstract existence result in \S\ref{sec:proof.I}, and check the validity of its
assumptions in \S\ref{sec:proof.II}.
\bigpar
From now on we rely on the following standing assumption.
\begin{assumption}
\label{Main.Ass}
$\Psi$ is twice continuously differentiable and strictly convex with
\begin{align*}
0<\ul{c}\leq\Psi^{\prime\prime}\at\zeta\leq\ol{c}<\infty
\end{align*}
for all $\zeta\in\Rset$ and some constants $\ul{c}$, $\ol{c}$.
\end{assumption}
We remark that Assumption~\eqref{Main.Ass} holds if and only if $\Phi$, the Legendre
transform of $\Psi$, is twice continuously differentiable and strictly convex with
$1/\ol{c}\leq\Phi^{\prime\prime}\leq1/\ul{c}$. This follows from basic properties of the
Legendre transform\fncite{DiB02}.
\bigpar
Throughout the remainder of this paper, $\fspaceL^2$ denotes the usual Hilbert space of
square-integrable and periodic functions on the periodicity cell $\La$, where the dual
pairing and the integral norm are normalized via
\begin{align*}
\skp{Q_1}{Q_2}=\fint_\La{Q_1\,Q_2}\,\dint\phase,
\qquad\norm{Q}_2^2=\fint_\La{Q}^2\,\dint\phase.
\end{align*}
Moreover, we define
\begin{align*}
\calH=\left\{Q\in\fspaceL^2\;:\;\fint_\La {Q}\,\dint\phase=0\right\}.%
\end{align*}
%
%
\subsection{Variational approach and abstract existence result}\label{sec:proof.I}
%
%
We now prove the existence of a one-parameter family of solutions
$\triple{Q}{\si}{\eta}\in\calH\times\Rset_+\times\Rset$ to the abstract wavetrain equation
\eqref{TWEquation.Abstract}. To this end we assume that the operator $\calB:\calH\to\calH$ is
compact and symmetric, and that $\Psi$ is normalized by
\begin{align}
\label{Normalised.Psi}
\Psi\at{0}=\Psi^\prime\at{0}=0.
\end{align}
Our variational approach to \eqref{TWEquation.Abstract} is based on the functionals
\begin{align}
\label{Def.Operators}
\calF\at{Q}=\tfrac{1}{2}\fint_\La{Q}\,\calB{Q}\dint\phase=\tfrac{1}{2}\skp{Q}{\calB{Q}}
,\qquad
\calW\at{Q}=\fint_\La\Psi\at{Q}\,\dint\phase,
\end{align}
and the constrained optimization problem
\begin{align}
\label{Optimization.Problem}
\text{maximize  }\calF\text{  on  }%
\calN_{\ga}=\big\{Q\in\calH\;:\;\calW\at{Q}\leq\ga\big\},
\end{align}
where $\ga>0$ is a free parameter. We readily justify that \eqref{TWEquation.Abstract} is the
corresponding Euler-Lagrange equation for a maximizer $Q\in\fspaceL^2$, where $\si$ and
$\mu=-\sigma\eta$ are Lagrangian multiplies for the constraints $\calW\at{Q}\leq\ga$ and
$\fint_\La{Q}\dint\phase=0$, respectively.
\par
The existence of maximizers $Q\in\calH$ can be proven by the \emph{direct method} using weak
compactness arguments, see the first part in the proof of
Theorem~\ref{Theo.Existence.Abstract}. However, in order to gain more qualitative information
about shape of the profile function $Q$ we refine the optimization problem
\eqref{Optimization.Problem}. For this purpose we introduce the corresponding (negative)
gradient flow, that is the $\calH$-valued ODE
\begin{align}
\label{Gradient.Flow}
\frac{\dint}{\dint\tau}{Q}&=\calB{Q}-\si\at{Q}\calP\at{Q},
\qquad
\calP\at{Q}=\Psi^\prime\at{Q}-\eta\at{Q}.
\end{align}
Here $\tau>0$ is the flow time, and
\begin{align*}
\eta\at{Q}=\fint_\La\Psi^\prime\at{Q}\dint\phase
,\qquad
\si\at{Q}=\frac{\skp{\calP\at{Q}}{\calB{Q}}}{\norm{\calP\at{Q}}_2^2}
\end{align*}
are two dynamical multipliers which guarantee that $\fint_{\La}Q\,\dint\phase=0$ and
$\calW\at{Q}=\const$ hold along each trajectory of \eqref{Gradient.Flow}. Notice that, by
construction, each stationary point of \eqref{Gradient.Flow} solves
\eqref{TWEquation.Abstract} with $\si=\si\at{Q}$ and $\eta=\eta\at{Q}$, and vice versa. For
later use we also mention that the explicit Euler scheme to \eqref{Gradient.Flow} with time
step $\tau>0$ is given by the mapping
\begin{align}
\label{Gradient.Flow.Discrete}
Q\mapsto\calI_\tau\at{Q}=\at{1-\tau}Q+\tau\calB{Q}-\tau\si\at{Q}\calP\at{Q}.
\end{align}
\bigpar
We are now able to formulate the refined existence result.%
\begin{theorem}
\label{Theo.Existence.Abstract}
Let $\calC\subset\calH$ be any positive cone that
\begin{enumerate}
\item
is convex and closed,
\item
invariant under the gradient flow \eqref{Gradient.Flow},
\item
satisfies $\sup_{Q\in\calC}\calF\at{Q}>0$.
\end{enumerate}
Then, for each $\ga>0$ the functional $\calF$ attains its maximum on $\calC\cap\calN_{\ga}$
and each maximizer $Q$ solves \eqref{TWEquation.Abstract} with multipliers $\eta=\eta\at{Q}$
and $\si=\si\at{Q}$. Moreover, each maximizer $Q$ satisfies $\si\at{Q}>0$ and
$\calW\at{Q}=\ga$.
\end{theorem}
In preparation for the proof of Theorem \ref{Theo.Existence.Abstract} we draw the following
conclusions from Assumption~\ref{Main.Ass}, the normalization condition
\eqref{Normalised.Psi}, and the postulated properties of $\calB$.
\begin{lemma}
\label{Lem:General.Props}
The following assertions are satisfied.
\begin{enumerate}
\item
$\calF:\fspaceL^2\to\Rset$ is well defined and weakly continuous,
\item
$\calW:\fspaceL^2\to\Rset$ is well defined, convex, continuous, and
G\^{a}teaux-differentiable with derivative $\partial\calW\at{Q}=\Psi^\prime\at{Q}$. We also
have
\begin{align}
\label{Lem:General.Props.Eqn1}
\tfrac{1}{2}\ul{c}\norm{Q}_2^2\leq\calW\at{Q}\leq\tfrac{1}{2}\ol{c}\norm{Q}_2^2
,\qquad%
\skp{\Psi^\prime\at{Q}}{Q}>0
\end{align}
for all $0\neq{Q}\in\fspaceL^2$.
\item
For each $\ga>0$, the set $\calN_{\ga}\subset\calH$ is convex and weakly compact in
$\fspaceL^2$. It is also star-shaped in the sense that for each $0\neq{Q}\in\calH$ there
exists a unique ${\la}_\ga\at{Q}>0$ such that $\la{Q}\in\calN_\ga$ for all
$0\leq\la\leq\la_\ga\at{Q}$.
\end{enumerate}
\end{lemma}
\begin{proof}
All assertions follow immediately from the definitions of $\calF$, $\calW$, and $\calN_\ga$,
see \eqref{Def.Operators} and \eqref{Optimization.Problem}.
\end{proof}
\begin{lemma}
\label{Lem:Gradient.Flow.Props}
The gradient flow \eqref{Gradient.Flow} is well defined on $\calH\setminus\{0\}$ and
conserves $\calW$. Moreover, $\calF$ increases strictly on each non-stationary trajectory.
\end{lemma}
\begin{proof}
Suppose that $Q$ is given with $\calP\at{Q}=0$, that means $\Psi^\prime\at{Q}=\const$. The
strict convexity of $\Psi$ implies $Q=\const$ and $Q\in\calH$ gives $Q=0$. From this we infer
that $\si\at{Q}$ is well defined for $Q\neq{0}$. Consequently, and since the right hand side
in \eqref{Gradient.Flow} is locally Lipschitz in $Q$, the ODE \eqref{Gradient.Flow} is well
posed in $\calH\setminus\{0\}$. By straight forward computations we now verify
$\tfrac{\dint}{\dint\tau}\calW\at{Q}=0$ and
\begin{align*}
\tfrac{\dint}{\dint\tau}\calF\at{Q}
=%
\skp{\calB{Q}-\si\at{Q}\calP\at{Q}}{\calB{Q}}
=
\norm{\calB{Q}}_2^2-\frac{\skp{\calP{Q}}{\calB{Q}}^2}{\norm{\calP\at{Q}}_2^2}
\geq0.
\end{align*}
In particular, we have $\tfrac{\dint}{\dint\tau}\calF\at{Q}=0$ if and only if $\calB{Q}$ and
$\calP\at{Q}$ are collinear, i.e., if and only if $Q$ is a stationary point of
\eqref{Gradient.Flow}.
\end{proof}
We now finish the proof of the abstract existence result.
\begin{Proof}[ of Theorem \ref{Theo.Existence.Abstract}]
With respect to the weak topology in $\fspaceL^2$, the objective functional $\calF$ is
continuous and the constraint set $\calC\cap\calN_{\ga}$ is compact. The existence of a
maximizer $Q$ thus follows from basic topological principles, and by assumption we have
$\calF\at{Q}>0$. By Lemma~\ref{Lem:General.Props} we find
\begin{align*}
\la_\ga\bat{Q}^2\calF\bat{Q}=\calF\bat{\la_\ga\nat{Q}Q}
\leq\calF\at{Q},
\end{align*}
which implies $\la_\ga\bat{Q}\leq1$. On the other hand, $Q\in\calN_\ga$ yields
$\la_\ga\bat{Q}\geq1$, and thus we find $\la_\ga\bat{Q}=1$, which means
$Q\in\partial\calN_\ga$. Moreover, Lemma~\ref{Lem:Gradient.Flow.Props} ensures that $Q$ is
stationary under the gradient flow \eqref{Gradient.Flow}, and hence a solution to
\eqref{TWEquation.Abstract}. Finally, testing \eqref{TWEquation.Abstract} with $Q$ we find
\begin{align*}
2\calF\at{Q}=\si\skp{\Psi^\prime\at{Q}}{Q},
\end{align*}
and $\calF\at{Q}>0$ combined with \eqref{Lem:General.Props.Eqn1} implies $\si>0$.
\end{Proof}
To conclude this section we mention that there also exist variational characterizations of
solutions $V$ to \eqref{TWEquation.Version1}, but these are less feasible than the integral
equation for the dual profile. The analogue to \eqref{TWEquation.Abstract} is
\begin{align*}
\si\calB^{-1}V=\Phi^\prime\at{v+V}-\theta,\qquad
\theta=\fint\Phi^\prime\at{v+V}\dint{\phase},
\end{align*}
which has a variational structure, but a direct treatment  is difficult as $\calB^{-1}$ is
not continuous. One can get rid off $\calB^{-1}$ by making the ansatz $V=\calB{S}$ with
$S\in\calH$ and working with the functionals
\begin{align*}
\calX\at{S}=\tfrac{1}{2}\fint_\La{S}\calB{S}\,\dint\phase,\qquad
\calY\at{S}=\fint_\La\Phi\at{v+\calB{S}}\,\dint\phase.
\end{align*}
However, the level sets of neither $\calX$ nor $\calY$ are weakly compact, and therefore it
is not obvious how to set up a variational framework that allows to prove the existence of
solutions. Finally, one might think about solving the equation
$V=\calB\at{\Phi^\prime\at{v+V}-\theta}/\si$ by fixed point arguments, but then one has find
a way to exclude the trivial solution $V\equiv{0}$.
%
%
\subsection{Proof of the main result}\label{sec:proof.II}
%
%
Here we apply the abstract result from the previous section and prove the existence of a
three-parameter family of wavetrains as claimed in the introduction. We therefore define
$\calC$ to be the positive cone of all $\fspaceL^2$-functions on $\La$ that are even,
unimodal, and have zero average. This reads
\begin{align}
\label{Definition.C.unimodal}
\calC = \Big\{Q\in\calH\;:\;
\text{$Q\at{-\phase_2}=Q\at{\phase_2}\leq{Q}\at{\phase_1}$ for almost all
$0\leq\phase_1\leq\phase_2\leq\pi$}\Big\}.
\end{align}
We next summarize some properties of the averaging operator $\calA_k$.
\begin{lemma}
\label{Lem.Operator.A.Props}
For all $0<k<\pi$ the operator $\calA_k$ maps $\calH$ into itself and is symmetric and
compact. Moreover, it maps $\calC$ into $\calC$.
\end{lemma}
\begin{proof}
The first three assertions follow directly from \eqref{Def.Operator.A}. It remains to prove
that $\calC$ is invariant under $\calA_k$ for $0<{k}\leq\pi/2$; the claim for
$\pi/2\leq{k}<\pi$ then follows from \eqref{Operator.A.Symmetry} and the fact that
$\calT:\calC\to-\calC$. Let $Q\in\calC$ be given, and notice that $\calA_k{Q}\in\calH$ is
even with weak derivative $P=\nabla_k{Q}\in\fspaceL^2$. For
$\phase\in\ccinterval{-\pi}{-\pi+k}$, the periodicity and the evenness of $Q$ imply
$Q\at{\phase-k}=Q\at{2\pi+\phase-k}=Q\at{-2\pi-\phase+k}$ where
$-\pi\leq-2\pi-\phase+k\leq\phase+k$, so the monotonicity of $Q$ in $\ccinterval{-\pi/2}{0}$
gives $2P\at\phase=Q\at{\phase+k}-Q\at{\phase-k}=Q\at{\phase+k}-Q\at{-2\pi-\phase+k}\geq0$.
Similarly, for $\phase\in\ccinterval{-\pi+k}{-k}$ and $\phase\in\ccinterval{-k}{0}$ we have
$2P\at\phase=Q\at{\phase+k}-Q\at{\phase-k}\geq0$ and
$2P\at\phase=Q\at{\phase+k}-Q\at{\phase-k}=Q\at{-\phase-k}-Q\at{\phase-k}\geq0$,
respectively. In summary, $\calA_k{Q}$ is increasing on $\ccinterval{-\pi/2}{0}$, and since
it is also even, the proof is complete.
\end{proof}
Lemma~\ref{Lem.Operator.A.Props} guarantees that both $\wh{\calB}_\kappa$ and
$\wt{\calB}_\kappa$ are symmetric and compact, and map $\calH\to\calH$ and $\calC\to\calC$.
\begin{lemma}
\label{Lem.Invariance.Props.B}
For given $q\in\Rset$ and $0<\kappa\leq\pi/2$, Theorem~\ref{Theo.Existence.Abstract} holds
with $\Psi_q$ instead of $\Psi$ if $\calC$ is given by \eqref{Definition.C.unimodal}, and if
$\calB$ is replaced by either $\wh{\calB}_\kappa$ or $\wt{\calB}_\kappa$.
\end{lemma}
\begin{proof}
The cone $\calC$ is convex and closed, and we have $\cos\in\calC$ with
\begin{math}
\nskp{\cos}{\wh{\calB}_\kappa\cos}=
\nskp{\cos}{\wt{\calB}_\kappa\cos}=\tfrac{1}{2}{\kappa}^{-1}{\sin\kappa}>0.
\end{math}
Therefore, it remains to show that $\calC$ is invariant under the gradient flow
\eqref{Gradient.Flow}. To this end we remark that $\calC$ is invariant under the action of
the operator $Q\mapsto{G}\at{Q}-\fint_{\La}G\at{Q}\dint\phase$ if and only if the function
$G:\Rset\to\Rset$ is non-decreasing. We now consider the Euler scheme of the gradient flow
\eqref{Gradient.Flow.Discrete} with time step $\tau>0$, that is
\begin{align*}
Q\mapsto\calI_\tau\at{Q}={G}_{\tau,\,\si\at{Q},\,\eta\at{Q}}\at{Q}+\tau\calB{Q},
\end{align*}
where the family of nonlinear functions $G_{\tau,\,\si,\,\eta}:\Rset\to\Rset$ is given by
\begin{align*}
G_{\tau,\,\si,\,\eta}\at\zeta=\zeta-\tau\zeta-\tau{\si}\Psi^\prime\at{\zeta}+\tau\sigma\eta.
\end{align*}
Thanks to Assumption~\ref{Main.Ass}, for each bounded set $U\subset\Rset^2$ there exists
$\bar\tau>0$ such that $\frac{\dint}{\dint\zeta}G_{\tau,\,\si,\,\eta}\at{\zeta}\geq0$ holds
true for all $0<\tau\leq\bar\tau$, all $\zeta\in\Rset$, and all $\pair{\si}{\eta}\in{U}$.
From this, the invariance of $\calC$ under the action of $\tau\calB$, and the continuous
dependence of $\si$ and $\eta$ on $Q\neq0$, we conclude that for each ball
$B\subset\fspaceL^2\setminus\{0\}$ there exists $\bar\tau>0$ such that $\calI_\tau$ maps
$B\cap\calC$ into $\calC$ for all $0<\tau\leq\bar\tau$. The claimed invariance of the
gradient flow now follows by passing to the limit $\tau\to{0}$.
\end{proof}
We are now able to concretize the informal existence result from the introduction as follows.
\begin{theorem}
\label{Theo:Main.Result}
Let $0<\kappa\leq\pi/2$, $q\in\Rset$, and $\ga>0$ be given. Then, there exist profiles
$\wh{Q},\wt{Q}\in\calC$ along with multipliers $\wh{\si},\wt{\si}>0$ and $\wh{\eta}$,
$\wt{\eta}$ such that
\begin{align*}
\wh{\si}\bat{\Psi^\prime_q\nat{\wh{Q}}-\wh{\eta}}=\wh{\calB}_\kappa\wh{Q},\qquad
\wt{\si}\bat{\Psi^\prime_q\nat{\wt{Q}}-\wt{\eta}}=\wt{\calB}_\kappa\wt{Q}.
\end{align*}
In particular, the profile $\wh{V}=\Psi^\prime_q\nat{\wh{Q}}-\wh{\eta}\in\calC$ solves
\begin{align*}
\wh{\om}\tfrac{\dint}{\dint\phase}\wh{V}=\nabla_\kappa\Phi^\prime\bat{\wh{v}+\wh{V}}
\end{align*}
with $\wh{\om}=\kappa\,\wh{\si}>0$ and $\wh{v}=\wh{\eta}+\Psi^\prime\bat{q}$, whereas
$\wt{V}=\Psi^\prime_q\nat{\wt{Q}}-\wt{\eta}\in\calC$ solves
\begin{align*}
\wt{\om}\tfrac{\dint}{\dint\phase}\wt{V}=\nabla_{\pi-\kappa}\Phi^\prime\bat{\wt{v}+\wt{V}}
\end{align*}
with $\wt{\om}=\kappa\,\wt{\si}>0$ and $\wt{v}=\wt{\eta}+\Psi^\prime\at{q}$.
\end{theorem}
\begin{proof}
The existence of both $\wh{Q}$ and $\wt{Q}$ is a consequence of
Lemma~\ref{Lem.Invariance.Props.B} and Theorem~\ref{Theo.Existence.Abstract}. The remaining
assertions then follow by straight forward computations.
\end{proof}
We conclude this section with some remarks concerning the assumptions and assertions of
Theorems \ref{Theo.Existence.Abstract} and \ref{Theo:Main.Result}.
\begin{enumerate}
\item
The strict convexity of $\Psi$ (or equivalently, of $\Phi$) is crucial for our existence
proof. In fact, it is truly necessary for relating the integral equations
\eqref{Dual.Problem.I} and \eqref{Dual.Problem.II} to the wavetrain equation
\eqref{TWEquation.Version1}. It also plays an important role in the proof of Theorem
\ref{Theo.Existence.Abstract} as it guarantees the weak compactness of $\calN_\ga$. The
assumptions about upper and lower bounds for $\Psi^{\prime\prime}$, however, are made for
convenience and might be weakened for the price of more technical effort.
\item
It would be highly desirable to give uniqueness results that classify wavetrains up to
phase shifts and scalings of the periodic cell. In analogy to the harmonic case, we
conjecture that the three-parameter family from Theorem\eqref{Theo:Main.Result} contains
\emph{all} wavetrains which satisfy the profile constraint $V\in\calC$, but we are not
able to prove this. Another open problem is the stability of wavetrains.
\item
Obviously, in Theorem \ref{Theo.Existence.Abstract} we can choose $\calC=\calH$ to obtain
solutions $Q\in\calN_\ga$ to the original optimization problem
\eqref{Optimization.Problem}. In this case, \eqref{TWEquation.Abstract} is provided by
the Lagrangian multiplier rule. If $\calC$, however, is a proper subset of $\calH$ then
the validity of \eqref{TWEquation.Abstract} does not follow from the multiplier rule but
is a consequence of the invariance properties of $\calC$. We also remark that numerical
simulations indicate that each maximizer of $\calF$ in $\calN_\ga$ is (up to phase
shifts) unimodal and even, and hence contained in $\calC$. We therefore conjecture that
\eqref{Optimization.Problem} and its refined variant from Theorem
\ref{Theo.Existence.Abstract} have the same solutions, but a rigorous proof is not
available.
\end{enumerate}
%
%
%
\section{Numerical simulations}\label{sec:num}
%
%
%
In this section we illustrate our analytical findings by numerical simulations, which are
computed by the following implementation of the discrete gradient flow
\eqref{Gradient.Flow.Discrete} with $\Psi_q$ instead of $\Psi$, see \eqref{Dual.Pot}, and
$\calB=\wh{\calB}_k$ and $\calB=\wt{\calB}_k$ for $k\in\ccinterval{0}{\pi/2}$ and
$k\in\ccinterval{\pi/2}{\pi}$, respectively:
\begin{figure}[t!]
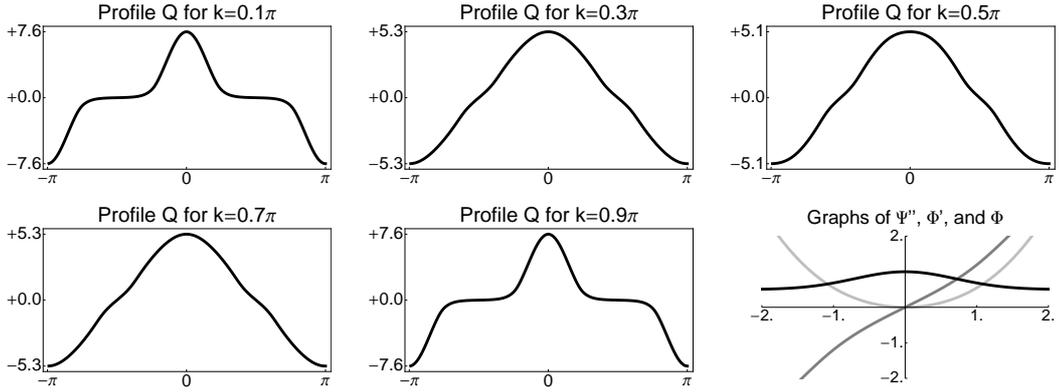
%
\centering{%
\renewcommand{\arraystretch}{\figstretch}%
\begin{tabular}{lcr}%
\begin{minipage}{\figwidth}%
\includegraphics[width=\textwidth, draft=\figdraft]%
{\figfile{ex_1_prof_1}}%
\end{minipage}%
&%
\begin{minipage}{\figwidth}%
\includegraphics[width=\textwidth, draft=\figdraft]%
{\figfile{ex_1_prof_2}}%
\end{minipage}%
&%
\begin{minipage}{\figwidth}%
\includegraphics[width=\textwidth, draft=\figdraft]%
{\figfile{ex_1_prof_3}}%
\end{minipage}%
\\%
\begin{minipage}{\figwidth}%
\includegraphics[width=\textwidth, draft=\figdraft]%
{\figfile{ex_1_prof_4}}%
\end{minipage}%
&%
\begin{minipage}{\figwidth}%
\includegraphics[width=\textwidth, draft=\figdraft]%
{\figfile{ex_1_prof_5}}%
\end{minipage}%
&%
\begin{minipage}{\figwidth}%
\hspace{.075\textwidth}%
\includegraphics[width=.925\textwidth, draft=\figdraft]%
{\figfile{ex_1_pot}}%
\end{minipage}%
\end{tabular}%
}%
\caption{%
Numerical simulations for example \eqref{Prms.Ex1} with several values of $k$. The plots show
the graph of the dual profile, i.e., $Q\at{\phase}$ against $\phase\in\La$. The lower right
picture sketches the graphs of $\Psi^{\prime\prime}$ (Black), $\Phi^{\prime}$ (Dark Grey),
and $\Phi$ (Light Grey).
}%
\label{Fig:Num1}%
\end{figure}%
\begin{figure}[t!]
\centering{%
\renewcommand{\arraystretch}{\figstretch}%
\begin{tabular}{lcr}%
\begin{minipage}{\figwidth}%
\includegraphics[width=\textwidth, draft=\figdraft]%
{\figfile{ex_2_prof_1}}%
\end{minipage}%
&%
\begin{minipage}{\figwidth}%
\includegraphics[width=\textwidth, draft=\figdraft]%
{\figfile{ex_2_prof_2}}%
\end{minipage}%
&
\begin{minipage}{\figwidth}%
\includegraphics[width=\textwidth, draft=\figdraft]%
{\figfile{ex_2_prof_3}}%
\end{minipage}%
\\%
\begin{minipage}{\figwidth}%
\includegraphics[width=\textwidth, draft=\figdraft]%
{\figfile{ex_2_prof_4}}%
\end{minipage}%
&%
\begin{minipage}{\figwidth}%
\includegraphics[width=\textwidth, draft=\figdraft]%
{\figfile{ex_2_prof_5}}%
\end{minipage}%
&
\begin{minipage}{\figwidth}%
\hspace{.075\textwidth}%
\includegraphics[width=.925\textwidth, draft=\figdraft]%
{\figfile{ex_2_pot}}%
\end{minipage}%
\end{tabular}%
}%
\caption{%
Numerical results for example \eqref{Prms.Ex2}.
}%
\label{Fig:Num2}%
\vspace{0.02\textwidth}%
\centering{%
\renewcommand{\arraystretch}{\figstretch}%
\begin{tabular}{lcr}%
\begin{minipage}{\figwidth}%
\includegraphics[width=\textwidth, draft=\figdraft]%
{\figfile{ex_3_prof_1}}%
\end{minipage}%
&%
\begin{minipage}{\figwidth}%
\includegraphics[width=\textwidth, draft=\figdraft]%
{\figfile{ex_3_prof_2}}%
\end{minipage}%
&%
\begin{minipage}{\figwidth}%
\includegraphics[width=\textwidth, draft=\figdraft]%
{\figfile{ex_3_prof_3}}%
\end{minipage}%
\\%
\begin{minipage}{\figwidth}%
\includegraphics[width=\textwidth, draft=\figdraft]%
{\figfile{ex_3_prof_4}}%
\end{minipage}%
&%
\begin{minipage}{\figwidth}%
\includegraphics[width=\textwidth, draft=\figdraft]%
{\figfile{ex_3_prof_5}}%
\end{minipage}%
&%
\begin{minipage}{\figwidth}%
\hspace{.075\textwidth}%
\includegraphics[width=.925\textwidth, draft=\figdraft]%
{\figfile{ex_3_pot}}%
\end{minipage}%
\end{tabular}%
}%
\caption{%
Numerical results for example \eqref{Prms.Ex3}.
}%
\label{Fig:Num3}%
\end{figure}%
\begin{figure}[t!]
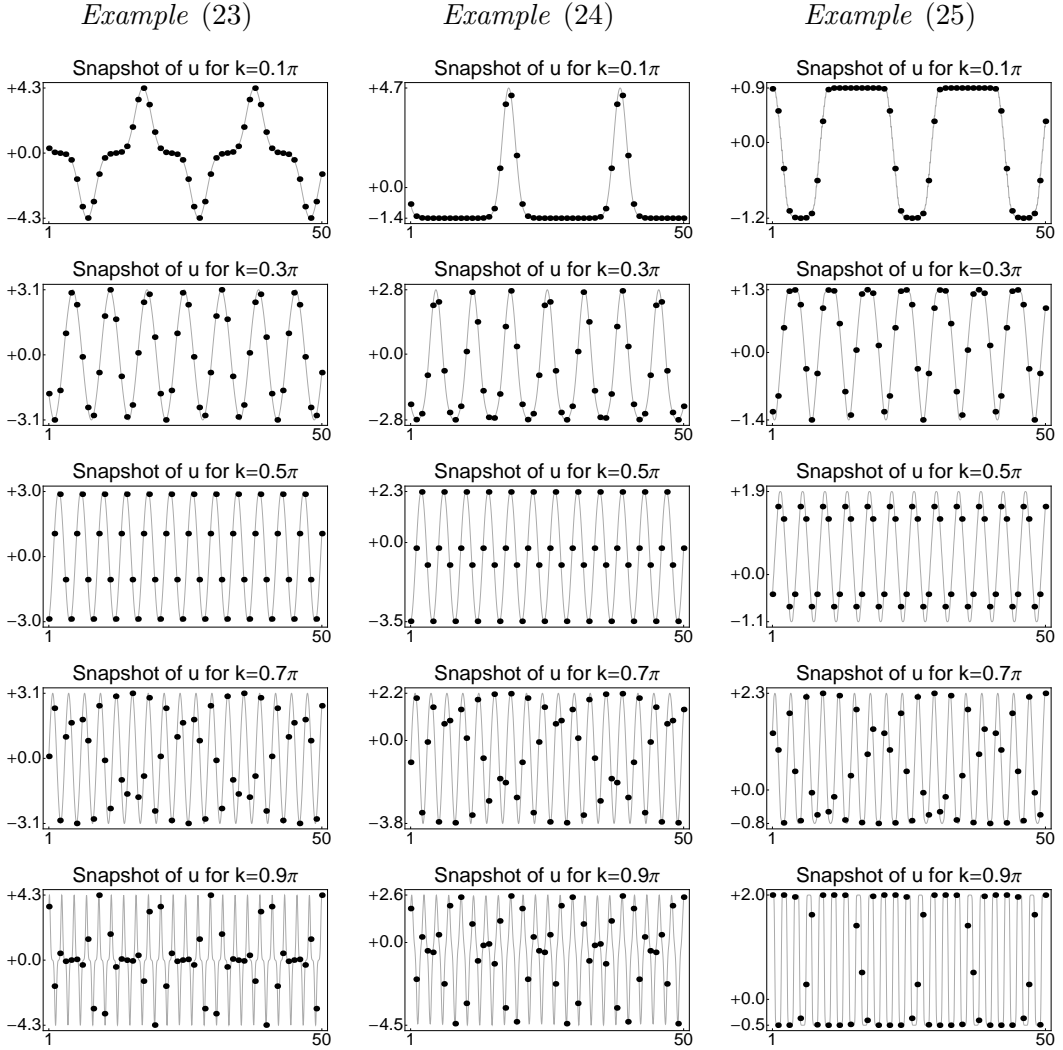
%
\centering{%
\renewcommand{\arraystretch}{\figstretch}%
\begin{tabular}{lcr}%
\begin{minipage}{\figwidth}%
\centering{\emph{Example \eqref{Prms.Ex1}}}%
\end{minipage}%
&%
\begin{minipage}{\figwidth}%
\centering{\emph{Example \eqref{Prms.Ex2}}}%
\end{minipage}%
&%
\begin{minipage}{\figwidth}%
\centering{\emph{Example \eqref{Prms.Ex3}}}%
\end{minipage}%
\vspace{-0.02\textheight}%
\\%
\begin{minipage}{\figwidth}%
\includegraphics[width=\textwidth, draft=\figdraft]%
{\figfile{ex_1_force_1}}%
\end{minipage}%
&%
\begin{minipage}{\figwidth}%
\includegraphics[width=\textwidth, draft=\figdraft]%
{\figfile{ex_2_force_1}}%
\end{minipage}%
&%
\begin{minipage}{\figwidth}%
\includegraphics[width=\textwidth, draft=\figdraft]%
{\figfile{ex_3_force_1}}%
\end{minipage}%
\\%
\begin{minipage}{\figwidth}%
\includegraphics[width=\textwidth, draft=\figdraft]%
{\figfile{ex_1_force_2}}%
\end{minipage}%
&%
\begin{minipage}{\figwidth}%
\includegraphics[width=\textwidth, draft=\figdraft]%
{\figfile{ex_2_force_2}}%
\end{minipage}%
&%
\begin{minipage}{\figwidth}%
\includegraphics[width=\textwidth, draft=\figdraft]%
{\figfile{ex_3_force_2}}%
\end{minipage}%
\\%
\begin{minipage}{\figwidth}%
\includegraphics[width=\textwidth, draft=\figdraft]%
{\figfile{ex_1_force_3}}%
\end{minipage}%
&%
\begin{minipage}{\figwidth}%
\includegraphics[width=\textwidth, draft=\figdraft]%
{\figfile{ex_2_force_3}}%
\end{minipage}%
&%
\begin{minipage}{\figwidth}%
\includegraphics[width=\textwidth, draft=\figdraft]%
{\figfile{ex_3_force_3}}%
\end{minipage}%
\\%
\begin{minipage}{\figwidth}%
\includegraphics[width=\textwidth, draft=\figdraft]%
{\figfile{ex_1_force_4}}%
\end{minipage}%
&%
\begin{minipage}{\figwidth}%
\includegraphics[width=\textwidth, draft=\figdraft]%
{\figfile{ex_2_force_4}}%
\end{minipage}%
&%
\begin{minipage}{\figwidth}%
\includegraphics[width=\textwidth, draft=\figdraft]%
{\figfile{ex_3_force_4}}%
\end{minipage}%
\\%
\begin{minipage}{\figwidth}%
\includegraphics[width=\textwidth, draft=\figdraft]%
{\figfile{ex_1_force_5}}%
\end{minipage}%
&%
\begin{minipage}{\figwidth}%
\includegraphics[width=\textwidth, draft=\figdraft]%
{\figfile{ex_2_force_5}}%
\end{minipage}%
&%
\begin{minipage}{\figwidth}%
\includegraphics[width=\textwidth, draft=\figdraft]%
{\figfile{ex_3_force_5}}%
\end{minipage}%
\end{tabular}%
}%
\caption{%
Snapshots of $u_j$ against $j=1\tdots50$ for the simulations from
Figure~\ref{Fig:Num1}--\ref{Fig:Num3} at a randomly chosen time $t$. Black points and grey
curves represent the data on the lattice $\Zset$ and the continuous carrier profiles $U=v+V$,
respectively. Under the evolution all points move with constant speed along the grey curves.
}
\label{Fig:Snapshots}%
\end{figure}%
\begin{enumerate}
\item
We sample the periodicity cell $\ccinterval{-\pi}{\pi}$ by grid points
$\phase_m=-\pi+2\pi{m}/M$ with $m=1\tdots{M}$.
\item
We approximate $Q$ by $Q_m=Q\at{\phase_m}$ and replace all integrals with respect to
$\phase$ by Riemann sums.
\item
After each update step according to \eqref{Gradient.Flow.Discrete}, we enforce the
conservation of $\calW\at{Q}$ by a scaling $Q_m\mapsto\la{Q}_m$, where the factor $\la$
is computed by a single Newton step.
\item
For given $k$ and $q$, we initialize the scheme with $Q_m=\alpha\cos\phase_m$, where the
amplitude $\alpha$ determines $\ga$ via
$\ga=\fint_\La\Psi_q\at{\alpha\cos\phase}\dint\phase$.
\end{enumerate}
Although this numerical scheme is rather simple it shows good convergence properties,
provided that $M$ is sufficiently large and the time step $\tau$ is sufficiently small. All
simulation presented below are performed with $\tau=0.1$ and $M=200$.
\bigpar
The first example concerns the data
\begin{align}
\label{Prms.Ex1}
\Psi^{\prime\prime}\at{\zeta}=\tfrac{1}{2}+\tfrac{1}{2}\exp\at{-\zeta^2}
,\qquad%
\alpha=5
,\qquad%
q=0,
\end{align}
where $\gamma$ is determined by $\alpha$ as described above. The numerical results are
presented in Figure~\ref{Fig:Num1}, which shows the dual profile $Q$ as function of $\phase$
for several values of $k$. We also refer to Figure~\ref{Fig:Snapshots}, which illustrates how
the wavetrains appear in snapshots $u_j=\Psi^\prime\at{q+Q\at{kj-\om{t}}}$ against
$j\in\Zset$. A special feature of this example is that $\Psi_q$ is even for $q=0$, and this
implies that \eqref{TWEquation.Abstract} is invariant under $Q\rightsquigarrow{-\calT{Q}}$.
We therefore find that the profiles are the same for $k$ and $\pi-k$, and exhibit the further
symmetry $Q\at{\phase+\pi}=-Q\at\phase$.
\par
The second example is computed with
\begin{align}
\label{Prms.Ex2}
\Psi^{\prime\prime}\at{\zeta}=1-\tfrac{1}{\pi}\arctan\at{2\zeta}
,\qquad%
\alpha=3
,\qquad%
q=0,
\end{align}
and illustrates the generic case that $\Psi_q$ is not even. In Figure \ref{Fig:Num2} we
therefore observe that the profiles do not satisfy $Q=-\calT{Q}$ anymore, and are no longer
the same for $k$ and $\pi-k$.
\par
Finally, the third example uses
\begin{align}
\label{Prms.Ex3}
\Psi^{\prime\prime}\at{\zeta}=\tfrac{1}{4}\exp\at{2\zeta}-\tfrac{1}{2}\zeta-\tfrac{1}{4}
,\qquad%
\alpha=2
,\qquad%
q=0,
\end{align}
and is shown in Figure~\ref{Fig:Num2}. Here $\Psi$ is still convex but does not satisfy
Assumption~\ref{Main.Ass} due to $\Psi^{\prime\prime}\at{0}=0$. Nevertheless, the numerical
scheme works very well and provides wavetrains for all values of $k$.
\appendix
%
%
\section{Lagrangian and Hamiltonian structures for scalar
conservation laws}\label{sec:appendix}
%
%
For the readers convenience we summarize some basic fact about the Lagrangian and Hamiltonian
structures of discrete scalar conservation laws. Both structures are very similar to the
respective structures for the KdV equation, compare for instance \cite{AM78}, and originate
from the skew-symmetry of centred difference operators. Specifically, for $\nabla$ with
$\nabla{y}_j=\tfrac12\at{y_{j+1}-y_{j-1}}$ the identity
\begin{align*}%
\sum_{j\in{\Zset}}y_j\nabla\tilde{y}_j=
-\sum_{j\in{\Zset}}\tilde{y}_j\nabla{y}_j
\end{align*}
holds for all sequences $y,\tilde{y}$ for which the series are well-defined. Notice that the
one-sided difference operators $\nabla^\pm$ with $\nabla^\pm{y}_j=\pm{y_{j\pm1}}\mp{y}_j$ are
not skew-symmetric but satisfy $\at{\nabla^\pm}^\ast=-\nabla^\mp$.
\par
To derive the \emph{Lagrangian structure} we assume that there exists $y=y_i\at{t}$ such that
$u_j=\nabla{y_j}$. This transform \eqref{Intro.Lattice} into
\begin{align}%
\label{App:Lattice.Integrated}
\nabla\dot{y}+\nabla\Phi^\prime\at{\nabla{y}}=0,
\end{align}%
and reveals that the \emph{action integral} is given by
\begin{align}%
\label{App:Lattice.ActionIntegral}
L\pair{\dot{y}}{y}
=%
-\tfrac{1}{2}\int_0^{t_\fin}\sum_{j\in\Zset}\dot{y}_j\nabla{y_j}\,\dint{t}%
-\int_0^{t_\fin}\sum_{j\in\Zset}\Phi\at{\nabla{y_j}}\,\dint{t},
\end{align}%
where $t_\fin$ is a given final time. In fact, computing the variational derivatives
$\partial_{\dot{y}}L=-\tfrac{1}{2}\nabla{y}$ and
$\partial_{y}L=\tfrac{1}{2}\nabla\dot{y}+\nabla\Phi^\prime\at{\nabla{y}}$ we easily check
that \eqref{App:Lattice.Integrated} is just the Euler-Lagrange equation
$-\tfrac{\dint}{\dint{t}}\partial_{\dot{y}}L+\partial_{y}L=0$.
\par
A special feature of $L$ is that the canonical momenta $\partial_{\dot{y}}L$ do not depend on
$\dot{y}$, and therefore we cannot apply the standard procedure to derive the Hamiltonian
structure. There is, however, a \emph{non-canonical Hamiltonian structure}. Using the
Hamiltonian
\begin{align}%
\notag
H\at{y}=\sum_{j\in\Zset}\Phi\at{\nabla{y_j}},
\end{align}%
the discrete scalar conservation law \eqref{Intro.Lattice} can be written as
\begin{align*}
\nabla\dot{y}=\partial_y{H}.
\end{align*}
This is indeed a Hamiltonian equation, where $\nabla$ acts as non-canonical symplectic
operator that corresponds to the symplectic product
\begin{align*}
\skp{\dot{y}}{y^\prime}_\sympl=-\sum_{j\in\Zset}\dot{y}_i\nabla{y_i^\prime}
=\sum_{j\in\Zset}y^\prime_i\nabla{\dot{y}_i}.
\end{align*}
The PDE \eqref{Intro.PDE1} has the same Lagrangian and Hamiltonian structures in the sense
that all formulas from above remain valid provided that we $\at{i}$ apply the scaling
\eqref{Intro.Scaling}, $\at{ii}$ replace sums over $j$ by integrals with respect to $\xi$,
and $\at{iii}$ employ the differential operator $D=\partial_\xi$ instead of $\nabla$. In
other words, the scalar conservation law \eqref{Intro.PDE1} is the Euler-Lagrange equation to
the action integral
\begin{align*}
L\pair{\partial\bar{y}}{\bar{y}}
=%
-\tfrac{1}{2}\int_0^{\tau_\fin}\int_\Rset
\partial_\tau\bar{y}D\bar{y}\,\dint{\xi}\dint{\tau} %
-\int_0^{\tau_\fin}\int_\Rset\Phi\at{D\bar{y}}\,\dint{\xi}\dint{\tau}.
\end{align*}%
It is moreover equivalent to
\begin{align*}
\partial_\tau{D}\bar{y}+{D}\Phi^\prime\at{{D}\bar{y}}=0,\qquad
\bar{u}=D\bar{y},
\end{align*}
which has Hamiltonian $H\at{\bar{y}}=\int_\Rset\Phi\at{D\bar{y}}\dint\xi$ and symplectic
product $\skp{\dot{y}}{y^\prime}_\sympl=\int_\Rset{}y^\prime{D}\dot{y}\dint\xi$. The
analogies between the structures of lattice and PDE are -- in view of the skew-symmetry of
both $\nabla$ and $D$ -- not surprising and exemplify the more general principles laid out in
\cite{GHM08}. Moreover, if we replace $D$ by the macroscopic difference operator
\begin{align*}
{D}_\eps\bar{y}\at{\xi}=\frac{\bar{y}\at{\xi+\eps}-\bar{y}\at{\xi-\eps}}{2\eps},
\end{align*}
we easily recover the scaled version of \eqref{Intro.Lattice}. In this sense the lattice
\eqref{Intro.Lattice} is just a \emph{variational integrator} of the scalar conservation law
\eqref{Intro.PDE1}, i.e., a semi-discrete scheme that respects the underlying variational and
symplectic structures.
\par
The very same idea also allows to derive variational integrators for the time discretization.
If we replace the continuous time derivative in \eqref{App:Lattice.ActionIntegral} by
$\tfrac{1}{2}h^{-1}\nat{y_{k+1,\,j}-y_{k-1,\,j}}$, where $h$ is the time step size and the
index $k$ refers to the $k^\text{th}$ time step, we obtain the discrete action integral
\begin{align*}%
L=%
-\frac{1}{8h}\sum_{k,\,j}\at{y_{k+1,\,j}-y_{k-1,\,j}}\at{y_{k,\,j+1}-y_{k,\,j-1}}
-\sum_{k,\,j}\Phi\at{\tfrac{1}{2}\at{y_{k,\,j+1}-y_{k,\,j-1}}},
\end{align*}
which depends on all the $y_{k,\,j}$'s. Evaluating the Euler-Lagrange equation
$\partial_{y_{k,j}}{L}=0$, and replacing $y$ by $u$, yields the fully discrete scheme
\begin{align}%
\label{App:Symplectic.Integrator}
u_{k+1,j}=u_{k-1,j}-h\Phi^\prime\at{u_{k,\,j+1}}+h\Phi^\prime\at{u_{k,\,j-1}},
\end{align}
which also served to compute the numerical examples from \S\ref{sec:intro}.
%
%
\paragraph*{Acknowledgements}%
This work was supported by the EPSRC Science and Innovation award to the Oxford Centre for
Nonlinear PDE
(EP/E035027/1).%
%
\providecommand{\bysame}{\leavevmode\hbox to3em{\hrulefill}\thinspace}
\providecommand{\MR}{\relax\ifhmode\unskip\space\fi MR }
\providecommand{\MRhref}[2]{%
  \href{http://www.ams.org/mathscinet-getitem?mr=#1}{#2}
} \providecommand{\href}[2]{#2}


\begin{thebibliography}{10}

\bibitem{AM78} R.~Abraham and J.E. Marsden, \emph{{F}oundations of {M}echanics}, 2. ed.,
  Perseus books, Cambridge Massachusetts, 1978, Updated 1985 Printing.

\bibitem{DM98} P.~Deift and T.-R. McLaughlin, \emph{A continuum limit of the {Toda} lattice},
  Mem. Americ. Math. Soc., vol. 131/624, American {Mathematical} {Society},
  1998.

\bibitem{DiB02} E.~DiBenedetto, \emph{Real {A}nalysis}, Birkh{\"a}user Advanced Texts: Basler
  Lehrb{\"u}cher, Birkh{\"a}user Boston, 2002.

\bibitem{DH08} W.~Dreyer and M.~Herrmann, \emph{Numerical experiments on the modulation
    theory
  for the nonlinear atomic chain}, Physica D \textbf{237} (2008), no.~2,
  255--282.

\bibitem{AMSMSP:DHR} W.~Dreyer, M.~Herrmann, and J.~Rademacher, \emph{Pulses, traveling waves
    and
  modulational theory in oscillator chains}, Analysis, Modeling and Simulation
  of Multiscale Problems (A.~Mielke, ed.), Springer, 2006.

\bibitem{El05} G.A. El, \emph{Resolution of a shock in hyperbolic systems modified by weak
  dispersion}, Chaos \textbf{15} (2005), 037103.

\bibitem{FV99} A.-M. Filip and S.~Venakides, \emph{Existence and modulation of traveling
    waves
  in particle chains}, Comm. Pure Appl. Math. \textbf{51} (1999), no.~6,
  693--735.

\bibitem{FP99} G.~Friesecke and R.L. Pego, \emph{Solitary waves on {FPU} \mbox{lattices}. {I.
  Qualitative} properties, renormalization and continuum limit}, Nonlinearity
  \textbf{12} (1999), no.~6, 1601--1627.

\bibitem{FW94} G.~Friesecke and J.A.D. Wattis, \emph{Existence theorem for solitary waves on
  lattices}, Comm. Math. Phys. \textbf{161} (1994), no.~2, 391--418.

\bibitem{GHM08} J.~Giannoulis, M.~Herrmann, and A.~Mielke, \emph{{L}agrangian and
    {H}amiltonian
  two-scale reduction}, J. Math. Phys. \textbf{49} (2008), no.~10, 103505, 42.

\bibitem{GL88} J.~Goodman and P.~Lax, \emph{On dispersive difference schemes}, Comm. Pure.
  Appl. Math. \textbf{41} (1988), 591--613.

\bibitem{Her09} M.~Herrmann, \emph{Periodic travelling waves in convex {K}lein-{G}ordon
  chains}, Nonlinear Anal.-Theory Methods Appl. \textbf{71} (2009), no.~11,
  5501--5508.

\bibitem{Her10c} \bysame, \emph{Heteroclinic standing waves in defocussing {DNLS} equations},
    to
  appear in Applicable Analysis, 2010.

\bibitem{Her10a} \bysame, \emph{Unimodal wavetrains and solitons in convex
  {F}ermi-{P}asta-{U}lam chains}, Proc. R. Soc. Edinb. Sect. A-Math.
  \textbf{140} (2010), no.~04, 753--785.

\bibitem{HR10b} M.~Herrmann and J.~D.~M. Rademacher, \emph{Heteroclinic travelling waves in
  convex {FPU}-type chains}, SIAM J. Math. Anal. \textbf{42} (2010), no.~4,
  1483--1504.

\bibitem{HR10a} \bysame, \emph{Riemann solvers and undercompressive shocks of convex {FPU}
  chains}, Nonlinearity \textbf{23} (2010), no.~2, 277--304.

\bibitem{HL91} T.Y. Hou and P.~Lax, \emph{Dispersive approximations in fluid dynamics}, Comm.
  Pure Appl. Math. \textbf{44} (1991), 1--40.

\bibitem{Ioo00} G.~Iooss, \emph{Travelling waves in the {F}ermi-{P}asta-{U}lam lattice},
  \mbox{Nonlinearity} \textbf{13} (2000), 849--866.

\bibitem{IJ05} G.~Iooss and G.~James, \emph{Localized waves in nonlinear oscillator chains},
  \mbox{Chaos} \textbf{15} (2005), 015113.

\bibitem{KvM75} M.~Kac and P.~van Moerbeke, \emph{On an explicitly soluble system of
    nonlinear
  differential equations related to certain {T}oda lattices}, Advances in Math.
  \textbf{16} (1975), 160--169.

\bibitem{Kam91} S.~Kamvissis, \emph{On the long time behavior of the double infinite toda
    chain
  under shock initial data}, Ph.D. thesis, New York University, 1991.

\bibitem{Lax86} P.D. Lax, \emph{On dispersive difference schemes}, Physica {D} \textbf{18}
  (1986), 250--254.

\bibitem{LL83} P.D. Lax and C.D. Levermore, \emph{The small dispersion limit of the
  {Korteweg}-de {Vries} equation, {I}, {II}, {III}}, Comm. Pure Appl. Math.
  \textbf{36} (1983), no.3, 253--290; no.5, 571--593; no.6, 809--830.

\bibitem{LLV93} P.D. {L}ax, C.D. {L}evermore, and S.~Venakides, \emph{The generation and
  propagation of oscillations in dispersive initial value problems and their
  limiting behavior}, Important developments in soliton theory (A.S. Fokas and
  V.E. Zakharov, eds.), Springer, 1993, pp.~205--241.

\bibitem{MX10} P.D. Miller and Zh. Xu, \emph{On the zero-dispersion limit of the
  {B}enjamin-{O}no {C}auchy problem for positive initial data}, preprint, see
  {\tt{arXiv:1002.3278}}, 2010.

\bibitem{PP00} A.~Pankov and K.~Pfl\"{u}ger, \emph{{T}raveling {W}aves in {L}attice
  {D}ynamical {S}ystems}, Math. Meth. Appl. Sci. \textbf{23} (2000),
  1223--1235.

\bibitem{SZ07} H.~Schwetlick and J.~Zimmer, \emph{Solitary waves for nonconvex {FPU}
  lattices}, J. Nonlinear Sci. \textbf{17} (2007), no.~1, 1--12.

\bibitem{SW97} D.~Smets and M.~Willem, \emph{Solitary waves with prescribed speed on infinite
  lattices}, J. Funct. Anal. \textbf{149} (1997), 266--275.

\bibitem{Str64} G.~Strang, \emph{Accurate partial difference methods. {II}. {N}on-linear
  problems}, Numer. Math. \textbf{6} (1964), 37--46.

\bibitem{Whi65} G.B. Whitham, \emph{Non-linear dispersive waves}, Proc. Roy. Soc. Ser. A
  \textbf{283} (1965), 238--261.

\bibitem{Whi74} \bysame, \emph{Linear and {Nonlinear} {Waves}}, Pure And Applied Mathematics,
  vol. 1237, Wiley Interscience, New York, 1974.

\end{thebibliography}
\end{document}